\documentclass[11pt,leqno,twoside]{amsart}
\usepackage{amssymb,amsmath,amsthm,soul,color}
\usepackage{t1enc}
\usepackage[cp1250]{inputenc}
\usepackage{a4,indentfirst,latexsym}
\usepackage{graphics}
\usepackage{mathrsfs}
\usepackage{cite,enumitem,graphicx}
\usepackage[colorlinks=true,urlcolor=blue,
citecolor=red,linkcolor=blue,linktocpage,pdfpagelabels,
bookmarksnumbered,bookmarksopen]{hyperref}
\usepackage[english]{babel}
\usepackage[left=2.61cm,right=2.61cm,top=2.72cm,bottom=2.72cm]{geometry}
\usepackage[metapost]{mfpic}
\usepackage[colorinlistoftodos]{todonotes}
\usepackage[normalem]{ulem}

\makeatletter
\providecommand\@dotsep{5}
\def\listtodoname{List of Todos}
\def\listoftodos{\@starttoc{tdo}\listtodoname}
\makeatother

\numberwithin{equation}{section}
\newtheorem{Th}{Theorem}[section]
\newtheorem{Prop}[Th]{Proposition}
\newtheorem{Lem}[Th]{Lemma}

\newtheorem{Rem}[Th]{Remark}

   \newcommand{\vp}{\varphi}
   
   \newcommand{\eps}{\varepsilon}

   \def\N{\mathbb{N}}
   \def\R{\mathbb{R}}

   \def\Qc{\mathcal{Q}}
   
   \def\X{\mathcal{X}}

   \def\RN{\mathbb{R}^N}
   
   \def\n{\nabla}
	\def\de{\partial}
\def\a{\alpha}
\def\b{\beta}
\def\l{\lambda}
\def\g{\gamma}
\def\s{\sigma}
\def\t{\theta}
\def\irn{\int_{\RN}}
\def\dis{\displaystyle}
\newcommand{\weakto}{\rightharpoonup}

\newcommand{\cC}{{\mathcal C}}

\title[Born-Infeld problem]{Born-Infeld problem with general nonlinearity}

\author[J. Mederski]{Jaros\l aw Mederski}
\author[A. Pomponio]{Alessio Pomponio}

\address{J. Mederski \hfill\break  
	\newline\indent Institute of Mathematics,
	\newline\indent Polish Academy of Sciences  
	\newline\indent ul. \'Sniadeckich 8, 00-656 Warsaw, Poland
	\newline\indent and
	\newline\indent Departement of Mathematics, Institute for Analysis
	\newline\indent Karlsruhe Institute of Technology (KIT)
	\newline\indent D-76128 Karlsruhe, Germany
}
\email{\href{mailto:jmederski@impan.pl}{jmederski@impan.pl}}

\address{A. Pomponio
\hfill\break  \newline\indent
Dipartimento di Meccanica, Matematica e Management
\newline\indent 
Politecnico di Bari
\newline\indent
Via Orabona 4,  70125  Bari, Italy}
\email{\href{mailto:alessio.pomponio@poliba.it}{alessio.pomponio@poliba.it}}

\thanks{A. Pomponio is partially supported by  PRIN 2017JPCAPN {\em Qualitative and quantitative aspects of nonlinear
PDEs}. J. Mederski is partially supported by the National Science Centre, Poland, Grant No. 2017/26/E/ST1/00817 and the Deutsche Forschungsgemeinschaft (DFG, German Research Foundation) – Project-ID 258734477 – SFB 1173}


\begin{document}
\begin{abstract} 
In this paper, using variational methods, we look for non-trivial solutions for the following problem
$$
\begin{cases}
-{\rm div}\left(a(|\nabla u|^2)\nabla u\right)=g(u), & \hbox{in }\mathbb{R}^N,\; N\geq 3,
\\[1mm]
u(x)\to 0, &\hbox{as }|x|\to +\infty,
\end{cases}
$$
under general assumptions on the continuous nonlinearity $g$. We assume only growth conditions of $g$ at $0$, however no growth conditions at infinity are imposed.  If $a(s)=(1-s)^{-1/2}$, we obtain the well-known Born-Infeld operator, but we are able to study also a general class of $a$ such that $a(s)\to+\infty$ as $s\to 1^{-}$.
We find a radial solution to the problem with finite energy.
\end{abstract}
\maketitle

\section{Introduction}

Almost a century ago, Born and Infeld introduced a new electromagnetic theory in a series of papers (see \cite{Bnat,B,BInat,BI})
as a nonlinear alternative to
the classical Maxwell theory. This theory was proposed to provide a model presenting a unitarian point of view to describe electrodynamics and had the notable feature to be a fine answer to the well-known {\em infinity energy problem}. In the Born-Infeld model, indeed, the electromagnetic field generated by a point charge has finite energy.
A crucial role is played by the following peculiar differential operator
\begin{equation*}
\Qc(u)=-{\rm div}\left(\dfrac{\n u}{\sqrt{1-|\n u|^2}}\right).
\end{equation*}
Such operator is present also in classical relativity, where it represents the mean curvature operator in Lorentz-Minkowski space, see for instance \cite{BS,CY}. 

In last years many authors focused their attention to problems related to  $\Qc$ in the whole $\RN$, with $N\ge 1$. In particular, some results for
\[
-{\rm div}\left(\dfrac{\n u}{\sqrt{1-|\n u|^2}}\right)=\rho, \qquad\hbox{ in }\RN,
\]
can be found in \cite{BCF,BDP,BDPR,Bon-Iac2,Bon-Iac3,H,K,K-corr}, under different assumptions on $\rho$. Here $\rho$ can be considered as an assigned charges source. See also \cite{APS}, where the Born-Infeld equation in coupled with the nonlinear Schr\"odinger one.

Few is still known, at contrary, in presence of a nonlinearity, namely for equations of this type
\begin{equation}\label{cong}
-{\rm div}\left(\dfrac{\n u}{\sqrt{1-|\n u|^2}}\right)=g(u), \qquad\hbox{ in }\RN.
\end{equation}

Let us observe that classical variational techniques do not work directly for this problem, due to the particular nature of the operator $\Qc$. Indeed, at least formally, solutions of \eqref{cong} are critical points of the functional
\[
I(u)=\irn \left(1-\sqrt{1-|\n u|^2}\right) -\irn G(u)\, dx,
\]
where $G$ is a primitive of $g$. However, since 
we have to impose the condition $|\n u|\le 1$, a.e. in $\RN$, the lack of regularity of the functional on the set $\{x\in \RN :|\n u|=1\}$ requires different and non-standard strategies.

One of the first paper dealing with this kind of problem using variational methods is \cite{BDD}, where $g(s)=|s|^{p-2}s$, for $p>2^*=\frac{2N}{N-2}$ and $N\ge 3$. By means of suitable truncation arguments (that will be crucial in our approach, as we will see later), the existence of {\em finite energy} solutions is proved.

We mention, moreover, \cite{A,A2,P} where \eqref{cong} has been studied by means of ODE-techniques finding solutions which could have infinite energy.  
In particular, in \cite{A,A2}, the existence of positive or sign-changing radial solutions is considered for a pure power nonlinearity or under suitable sign assumptions on $g$ (a prototype of such nonlinearity is  $g(s)=-\l s+s^p$, for $\l>0$ and $p>1$). In \cite{P}, instead, the existence of oscillating solutions of \eqref{cong}, namely with an unbounded sequence of zeros, is proved for nonlinearities such that $g'(0) > 0$.  
Finally, in \cite{BCP}, a similar problem is considered in an exterior domain. 

Our aim is to show existence of finite energy radial solutions involving a large class of operators and nonlinearities  in the spirit of Berestycki and Lions \cite{BerLions,BerLionsII} and we will present an adequate variational approach for the problem.
 More precisely we consider
\begin{equation}\label{eq}
\begin{cases}
-{\rm div}\left(a(|\n u|^2)\n u\right)=g(u), & \hbox{in }\RN,\; N\geq 3,
\\[5mm]
u(x)\to 0, &\hbox{as }|x|\to +\infty,
\end{cases}
\end{equation}
under the following assumptions on $a$:
\begin{enumerate}[label=(a\arabic{*}), ref=a\arabic{*}]
	\setcounter{enumi}{-1}
	\item \label{a0}$a:[0,1) \to (0,+\infty)$ is continuous, of class $\cC^1$ on $(0,1)$, and $[0,1) \ni s\mapsto a(s)s$ is strictly convex;
	\item \label{a1} 
	$\displaystyle\lim_{s\to 1^-}a(s)=+\infty;$
\end{enumerate}
and on the nonlinearity $g$:
\begin{enumerate}[label=(g\arabic{*}), ref=g\arabic{*}]
\setcounter{enumi}{-1}
	\item \label{g0}$g:\R \to \R$ is continuous and odd;
	\item \label{g1} for some $\gamma \geq 2^*/2$, we have
$$-\infty<\displaystyle\liminf_{s\to 0}\frac{g(s)}{|s|^{\gamma-1}}
\leq\displaystyle\limsup_{s\to 0}\frac{g(s)}{|s|^{\gamma-1}}=-m<0;$$
	\item \label{g3}there exists  $\xi_0>0$ such that $G(\xi_0)>0$, where 
	$$G(s)=\int_0^s g(t)\, dt,\quad\hbox{for }s\in\R.$$
\end{enumerate}
Clearly, $a(s)=(1-s)^{\alpha}$ with $\alpha<0$ satisfies \eqref{a0},  \eqref{a1}, and we get the operator $\Qc$ for $\alpha=-1/2$. Another important example is the following general mean curvature operator arising in the study of hypersurfaces in the Lorentz–Minkowski space $\mathbb{L}^{N+1}$ and in $\R^{N+1}$ given by
\begin{equation}\label{eq:GMC}
a(s):=\beta (1-s)^{-1/2}-\gamma (1+s)^{-1/2},\quad \beta>0, \gamma\geq 0,
\end{equation}
see  \cite{CY,Kob,Dai}  and references therein.

With regard to $g$,  by assumption \eqref{g1}, the problem is in the so called {\em positive mass case}.
%, according with the definition given in \cite{BerLions}. 
We will consider also the {\em zero mass case} namely, instead of \eqref{g1}, we will assume
\begin{enumerate}[label=(g\arabic{*}$'$), ref=g\arabic{*}$'$]
	\setcounter{enumi}{0}
	\item \label{g1z} for some $\gamma>2^*$, we have
	$$-\infty<\displaystyle\liminf_{s\to 0}\frac{g(s)}{|s|^{\gamma-1}}
	\leq\displaystyle\limsup_{s\to 0}\frac{g(s)}{|s|^{\gamma-1}}=0.$$
\end{enumerate}
%\begin{enumerate}[label=(g\arabic{*}$'$), ref=g\arabic{*}$'$]
%\setcounter{enumi}{0}
%	\item \label{g1z}$\displaystyle\limsup_{s\to 0}g(s)/|s|^{\gamma-1}= 0$ and $\displaystyle\limsup_{s\to 0}|g(s)|/|s|^{\gamma-1}<+\infty$, for some $\gamma > 2^*$.
%\end{enumerate}
If the constant $\g$ in the assumption \eqref{g1z} is not greater than $N$, we need also a condition at infinity on $g$. More precisely, we require
\begin{enumerate}[label=(g\arabic{*}$''$), ref=g\arabic{*}$''$]
\setcounter{enumi}{0}
	\item \label{g1zinf}whenever $N\geq \g>2^*$,  $\displaystyle\limsup_{s\to +\infty}g(s)/|s|^{q^*-1}= 0$, for some $q\in \left(\frac{N\g}{N+\g},N\right)$,
\end{enumerate}
where $q^*=\frac{qN}{N-q}$. Observe that, clearly, we have $2^*<\g<q^*$ and
it is easy to see that a pure power non-linearity $g(s)=|s|^{p-2}s$, with $p>2^*$, satisfies assumptions \eqref{g1z} and \eqref{g1zinf}. Therefore we generalize the existence results contained in \cite{BDD}.

%{\color{blue} In \eqref{g1} and \eqref{g1z} we assume $\gamma \geq 2^*/2$. Actually this is a quite technical requirement which is essential only in the final estimates, while, in all the other our arguments, we can simply take $\gamma>1$.}

We recall that these kinds of hypotheses on $g$ have been introduced for the first time in \cite{BerLions,BerLionsII} for the study of 
\begin{equation}\label{eqBL}
-\Delta u=g(u), \qquad\hbox{ in }\RN,
\end{equation}
where $\gamma=2$.
However, we want to remark that, in contrast to what happens in these previous papers, in our case there is no assumption on the behaviour at infinity of $g$ in the positive mass case or in the zero mass case if, in \eqref{g1z}, $\gamma>N$.
This is a direct consequence of the natural framework associated to \eqref{eq} which has to take in account the condition $|\n u|\le 1$, a.e. in $\RN$: this assures that each function is, actually, bounded. See Section \ref{seff} for more details. 

An intermediate step  for the study of \eqref{eq}, based on an approximation argument, has been widely studied in the literature, e.g. see \cite{AW} and references therein. Indeed by the Taylor expansion of $\frac{1}{\sqrt{1-|u|}}$ to the $k$-th order, we  
arrive at the approximated problem
\begin{equation} \label{eqk}
\Qc(u)\approx -\Delta u-\frac12\Delta_4u -\frac{3}{2\cdot2^2} \Delta_6 u 
-\cdots -\frac{(2k-3)!!}{(k-1)!\cdot2^{k-1}}\Delta_{2k} u =g(u) \quad \hbox{in} \ \RN.
\end{equation}
Note that  \cite{AW} deals precisely with \eqref{eqk}, where $g$ satisfying more restrictive Berestycki-Lions-type assumptions. In  \cite{AW}  (see also the references therein), it is not clear if one can solve \eqref{cong} passing to the limit, as $k\to +\infty$. We would like to mention that some partial results using this approximation process have been obtained only in case of the fixed charges source $\rho$ on the right hand side instead of the nonlinear term $g(u)$, see e.g. \cite{BDP,BDPR,K,K-corr}.
Therefore \eqref{cong} requires a different variational approach presented in this work.

Our main result reads as follows.
\begin{Th}\label{main}
Suppose that $a$ satisfies \eqref{a0}, \eqref{a1} and  $g$ satisfies \eqref{g0} and \eqref{g3}. If, in addition,  \eqref{g1} holds, or $\gamma>N$ and \eqref{g1z} holds, or    $\gamma\leq N$ and  both \eqref{g1z}, \eqref{g1zinf} hold, then there exists a nontrivial radial solution $u$ of \eqref{eq} such that 
\[
\irn A(|\n u|^2)\,dx, \irn a(|\n u|^2)|\n u|^2\,dx, \irn |G(u)|\, dx<+\infty,
\]
where $A(s)=\int_0^s a(t)\,dt$. 
\end{Th}

We use a truncation argument applied to $a$ similarly as in \cite{BDD} but  due to the lack of scaling of the nonlinearity we use a different variational approach for \eqref{eq}. Inspired by \cite{HIT,jj} (see also \cite{AMP,ADP,CDPS,DMP}), we
will adapt for our problem the method explored considering an auxiliary functional that allows to construct a suitable Palais-Smale sequence, which almost satisfies a Pohozaev type identity. The compactness properties of the general nonlinear term will be investigated similarly as in \cite{MederskiBL,MederskiBL2}, see Sections \ref{se>} and \ref{se0} for more details.

The paper is organized as follows. In Section \ref{seff} we introduce our functional framework and some technical tools. Section \ref{se>} and Section \ref{se0} will be devoted, respectively, to the positive mass case and to the zero mass one and, therein, we will prove our main result.

We conclude this introduction fixing some notations. For any $p\ge 1$, we denote by $L^p(\RN)$ the usual Lebesgue spaces equipped by the standard norm $|\cdot|_{p}$.
In our estimates, we will frequently denote by $C>0$, $c>0$ fixed
constants, that may change from line to line, but are always
independent of the variable under consideration. 
We also use the notation $o_n(1)$ to indicate a quantity which goes to zero as $n\to +\infty$.
Moreover, for any $R>0$, we denote by $B_R$ the ball of $\RN$ centred in the origin with radius $R$.
Finally, if $u$ is a radial function of $\RN$, with an abuse of notation, for any $x\in \RN$, we denote $u(x) = u(r)$, with $r = |x|$. 

\subsection*{Acknowledgements} 
This work has been  partially carried out during a stay of J.M. at Karlsruhe Institute of Technology. This work has been also partially carried out during a stay of A.P. in Poland at Nicolaus Copernicus University in Toruń, and at Institute of Mathematics of the Polish Academy of Sciences in Warsaw. J.M and A.P. would like to express their deep gratitude to these prestigious institutions  for the support and warm hospitality. 

The authors wish to thank Prof. Antonio Azzollini for many inspiring comments and discussions.

\section{Functional framework}\label{seff}

In this section we introduce the functional framework related to \eqref{eq} with some useful continuous and compact embedding properties. Moreover, following \cite{BDD}, we present a truncated problem which will play a crucial role in our arguments.

Take any $q>2$. Let $\X^{2,q}_0$
 be the completion of $\cC_0^{\infty}(\R^N)$ with respect to the following norm
 $$\|u\|_0=\big(|\nabla u|_2^2+|\nabla u|_q^2\big)^{1/2}.$$
Recall that 
%$\X_0^{2,q}$ is continuously embedded into $L^p(\R^N)$ for $p\in [2^*,+\infty]$ 
\[
\X_0^{2,q} \hookrightarrow L^p(\R^N),\qquad\text{for }
p\in 
\begin{cases}
[2^*,q^*]& \text{if }q<N,
\\
[2^*,+\infty)& \text{if }q=N,
\\
[2^*,+\infty]& \text{if }q>N,
\end{cases}
\]
and, denoting 
$$\X_0:=\X_{0,{\rm rad}}^{2,q}=\big\{u\in \X^{2,q}_0: u \hbox{ radially symmetric}\big\},$$
we have
\[
\X_0 \hookrightarrow \hookrightarrow L^p(\R^N),\qquad\text{for }
p\in 
\begin{cases}
(2^*,q^*)& \text{if }q<N,
%\\
%(2^*,+\infty)& \text{if }q=N,
\\
(2^*,+\infty)& \text{if }q\geq N,
\end{cases}
\]
%embeds compactly into $L^p(\R^N)$, for $p\in (2^*,+\infty)$, 
see e.g. \cite{BDD,AW}. Moreover,
as in \cite{BDD,S}, we have the following
\begin{Lem}\label{le:strauss}
Let $p\in [2,q]$, if $q<N$, and  $p \in [2,N)$, if $q\ge N$. Then there exists $C > 0$ (depending only on $N$ and $p$) such that for all $u \in \X_0$, there holds
\[
|u(x)| \le C|x|^{-\frac{N-p}p} |\n u|_p, 
\]
for almost every $x\in \RN\setminus \{0\}$.
\end{Lem}

In the positive mass case we always assume that $q>N$ and let $\X^{2,q,\gamma}$
 be the completion of $\cC_0^{\infty}(\R^N)$ with respect to the following norm
 $$\|u\|=\big(|\nabla u|_2^2+|\nabla u|_q^2+|u|_{\gamma}^2\big)^{1/2}$$
 and, clearly, if $\gamma\geq 2^*$, then $\X^{2,q,\gamma}$ and $\X^{2,q}_0$ coincides. Moreover
$\X^{2,q,\gamma}$ is continuously embedded into $L^p(\R^N)$ for $p\in[\min\{2^*,\gamma\},+\infty]$ and
$$\X:=\X_{\rm rad}^{2,q,\gamma}=\big\{u\in \X^{2,q,\gamma}: u \hbox{ radially symmetric}\big\}$$  
embeds compactly into $L^p(\R^N)$, for $p\in (\min\{2^*,\gamma\},+\infty)$.

Similarly as in  \cite{BDD} for $\Qc$ we introduce a truncated problem.
Let us fix  $\t_1\in (0,1)$. For any $\t\in (0,\t_1]$ we fix $q=q(\t)>N$ such that
\begin{equation}\label{eq:convexA}
q\geq 2\frac{a'(1-\theta)(1-\theta)+a(1-\theta)}{a(1-\theta)}.
\end{equation}
Then we define  a continuous function $a_\t:[0,+\infty)\to \R^+$ by
\[
a_\t(s):=
\begin{cases}
a(s)& \hbox{ if }0\le s\le 1-\t,
\\
(1-\theta)^{-\frac{q-2}{2}} a(1-\theta) s^\frac{q-2}{2} & \hbox{ if } s> 1-\t.
\end{cases}
\]
The functions $a_\t(s)$ and $\vp(s):=a_\t(s)s$ are differentiable in $ [0,+\infty)\setminus\{1-\t\}$ and, by \eqref{eq:convexA} and \eqref{a0}, we deduce that 
$\vp'(s_1)< \vp'_-(1-\t)\le \vp'_+(1-\t)< \vp'(s_2)$, for any $s_1<1-\t<s_2$. 

\begin{Lem}\label{lemConvex}
The map $\vp(s)$ is strictly convex.
\end{Lem}
\begin{proof}
Clearly $\vp$ is strictly convex on $[0,1-\theta]$ and on $[1-\theta,+\infty)$. Take $0<s<1-\theta<t$. If $\frac{s+t}{2}\le 1-\theta$, then by the convexity we obtain
\begin{align*}
\vp(s)-\vp\Big(\frac{s+t}{2}\Big)&> \vp'\Big(\frac{s+t}{2}\Big)\Big(s-\frac{s+t}{2}\Big),\\
\vp(1-\theta)-\vp\Big(\frac{s+t}{2}\Big)&> \vp'\Big(\frac{s+t}{2}\Big)\Big(1-\theta-\frac{s+t}{2}\Big),\\
\vp(t)-\vp(1-\theta)&> \vp'_+(1-\theta)(t-1+\theta).\\
\end{align*}
In view of \eqref{eq:convexA} we get $\vp'_+(1-\theta)\ge \vp'\big(\frac{s+t}{2}\big)$ and we conclude 
$$\frac{\vp(s)+\vp(t)}{2}> \vp\Big(\frac{s+t}{2}\Big).$$
Similarly we argue if $\frac{s+t}{2}>1-\theta$ and we conclude.
\end{proof}

For the positive mass case we will consider the following truncated problem
\begin{equation}
\label{eqTheta}
\left\{
\begin{array}{ll}
-{\rm div}\left(a_\t(|\n u|^2)\n u\right) u  = g(u)&\quad \hbox{in }\R^N,\\
u\in \X.&%\quad\hbox{as }|x|\to\infty.
\end{array}
\right.
\end{equation}
For the zero mass case, instead, we will consider the following truncated problem
\begin{equation}
\label{eqThetaz}
\left\{
\begin{array}{ll}
-{\rm div}\left(a_\t(|\n u|^2)\n u\right) u  = g(u)&\quad \hbox{in }\R^N,\\
u\in \X_0.&%\quad\hbox{as }|x|\to\infty.
\end{array}
\right.
\end{equation}
Clearly, if $u_\t$ is a solution of \eqref{eqTheta} or of \eqref{eqThetaz} such that $|\n u_\t|\le 1-\t$, then $u_\t$ is a solution also of \eqref{eq}.

Observe that there exists $\bar c_\t=\bar c_\t(\t)>0$ such that
\begin{align}
\label{a2q}
\bar c \left(s^2+|s|^q\right)
&\le a_\t(s^2)s^2\le \bar c_\t \left(s^2+|s|^q\right), \qquad\hbox{ for all }s\in \R,\\		
\label{A2q}
\bar c \left(s^2+|s|^q\right)
&\le A_\t(s^2)\le \bar  c_\t \left(s^2+|s|^q\right), \hspace{3mm}\qquad\hbox{ for all }s\in \R,
\end{align} 
where $A_\t(s)=\int_0^s a_\t(t)\,dt$ and
$$\bar c:=\frac{2}{q}\cdot\frac{(1-\theta_1)^{\frac{q-2}{2}}}{1+(1-\theta_1)^{q-2}}\cdot \min_{s\in[0,1)}a(s)$$
is independent of $\theta$.

We conclude this section with the following lemma, which is also new for $\Qc$ and which will play a crucial role in our arguments.

\begin{Lem}\label{lem:convgrad}
Suppose that $u_n\weakto u_0$ in $\X_0$ and
	\begin{align}\label{eq:lemmaconv1}
	\lim_{n \to +\infty}\irn a_\t(|\n u_n|^2)|\n u_n|^2\, dx
	=\irn a_\t(|\n u_0|^2)|\n u_0|^2\, dx .
	\end{align}
Then $u_n \to u_0$ strongly in $\X_0$.
\end{Lem}
\begin{proof}
Let $\vp:\R^N\to \R$ be given by $\vp(v):=a_\t(|v|^2)|v|^2$, for $v\in \R^N$. By Lemma \ref{lemConvex}, $\vp$ is strictly convex, hence
 the map $\Phi:\X_0 \to \R$, such that 
\[
\Phi(u):=\int_{\R^N}\vp(\nabla u)\, dx, \qquad\hbox{ for }u\in \X_0,
\]
is well defined and strictly convex as well. So, since $\frac12 (\nabla u_n+\nabla u_0)\weakto \nabla u_0$, we obtain
\begin{equation}\label{liminf}
\liminf_{n\to+\infty}\int_{\R^N}\vp\Big(\frac12 (\nabla u_n+\nabla u_0)\Big)\, dx\geq \int_{\R^N}\vp(\nabla u_0)\, dx.
\end{equation}
Then, taking into account the convexity of $\vp$, we know that, a.e. in $\RN$,
$$\xi_n:=\frac12\big(\vp(\nabla u_n)+\vp(\nabla u_0)\big)-\vp\Big(\frac12 (\nabla u_n+\nabla u_0)\Big)\geq 0,$$
hence, by \eqref{eq:lemmaconv1} and \eqref{liminf},
\begin{equation}\label{eq:limxi}
\limsup_{n\to+\infty}\int_{\R^N}\xi_n\, dx=0.
\end{equation}
For any $k\geq 1$ we define
\begin{align*}
\mu_k&:=\inf\left\{\frac12\big(\vp(v_1)+\vp(v_2)\big)-\vp\Big(\frac12 (v_1+v_2)\Big): v_1,v_2\in\R^N\hbox{ s.t. }|v_1|, |v_2|\leq k, |v_1-v_2|\geq \frac1k\right\},\\
\Omega_{n,k}&:=\left\{x\in\R^N: |\nabla u_n|, |\nabla u_0|\leq k, |\nabla u_n-\nabla u_0|\geq \frac1k\right\}.
\end{align*}
Since $\mu_k>0$, by the strict convexity of $\vp$, and \eqref{eq:limxi} holds, we infer that the Lebesgue measure $|\Omega_{n,k}|\to 0$, as $n\to +\infty$. Take any $\eps>0$, we find a subsequence $\{n_k\}$ such that
$|\bigcup_{k=1}^{\infty} \Omega_{n_k,k}|<\eps$. Again letting $\eps\to0$ and passing to a subsequence we obtain that $\nabla u_n\to \nabla u_0$ a.e. on $\R^N$. Note that $a_\theta$ is of class $\cC^1$ on $(0,1-\theta)$ and $(1-\theta,+\infty)$, hence $\vp'$ exists almost everywhere. Now take $s\in [0,1]$, by \eqref{a2q} we observe that the sequence
$\{\vp'(\nabla u_n-s\nabla u_0)\nabla u_0\}$ is uniformly integrable and tight and converges a.e. to $\vp'\big((1-s)\nabla u_0\big)\nabla u_0$. In view of the Vitali Convergence Theorem we get 
\begin{align*}
\int_{\R^N}\vp(\nabla u_n)\,dx-\int_{\R^N}\vp(\nabla u_n-\nabla u_0)\,dx&=\int_0^1\int_{\R^N}\vp'(\nabla u_n-s\nabla u_0)\nabla u_0\,dx \,ds
\\
&\xrightarrow[n\to +\infty]{} \int_0^1\int_{\R^N}\vp'\big((1-s)\nabla u_0\big)\nabla u_0\,dx \,ds
\\
&=\int_{\R^N}\vp(\nabla u_0)\,dx.
 \end{align*}
Since \eqref{eq:lemmaconv1} holds, we get 
$$\int_{\R^N}\vp(\nabla u_n-\nabla u_0)\,dx\to 0,$$
as $n\to +\infty$, and by \eqref{a2q} we conclude.
\end{proof}

\section{The positive mass case}\label{se>}

In this section we deal with the positive mass case, namely, we will assume on $g$ \eqref{g0}, \eqref{g1} and \eqref{g3}.

Let 
%\footnote{\color{orange}In order to have a positive solution, we could define
%\[
%g_1(s):=
%\begin{cases}
%\max\{g(s)+ms^{\gamma-1},0\} &\hbox{for }s\geq 0,
%\\
%0 &\hbox{for }s< 0.
%\end{cases}
%\]
%and $g_2(s):=g_1(s)-g(s)$ for $s\geq 0$ and then we can extend it as an odd function for $s<0$. All the equations from \eqref{eq:NewCond1} to \eqref{Gbarc} should still hold. 
%What do you think?}
$g_1(s):=\max\{g(s)+ms^{\gamma-1},0\}$, for $s\geq 0$, 
and $g_2(s)=g_1(s)-g(s)$, for $s\geq 0$, and $g_i(s)=-g_i(-s)$ for $s<0$.
Then $g_1(s),g_2(s)\geq 0$, for $s\geq 0$,
\begin{align}\label{eq:NewCond1}
\lim_{s\to 0} g_1(s)/s^{\gamma-1}&=0, 
\\
g_2(s)&\geq m s^{\gamma-1}, \quad \hbox{ for }s\geq 0.\label{eq:NewCond2}
\end{align}
%{\color{red}Moreover, by some computations, we have that for any $\eps>0$ there exists
%$C_\eps>0$ such that
%    \begin{equation}
%        g_1(s) \le C_\eps \left(e^{\alpha s^{\gamma_2}}-1\right)+\eps g_2(s),\quad  \forall
%        s\ge0\label{eq:g1g2}.
%    \end{equation}}
If we set
$$G_i(s)=\int_0^s g_i(t)\, dt,\quad \hbox{ for }i=1,2,$$ 
then by \eqref{eq:NewCond2} 
we have
    \begin{equation}
         G_2(s) \ge \frac m{\g} |s|^{\g},\quad  \hbox{ for } s\in\R.\label{eq:G2}
    \end{equation}
%{\color{red}and for any $\eps>0$ there exists $C_\eps>0$ such that
%    \begin{equation}
%        G_1(s) \le C_\eps \left(e^{\alpha |s|^{\gamma_2}}-1\right)+\eps G_2(s),\quad  \forall
%        s\in\R\label{eq:G1G2}.
%    \end{equation}
%Finally observe that by \eqref{g1} and  \eqref{g2'},we have that
%\begin{equation}\label{esti_su-g}
%|g(s)|\le c_1|s|^{\g-1}+c_2(e^{\a |s|^{\g_2}}-1),
%\end{equation}
%for any $s\in \R$.
%}

By \eqref{g1} and \eqref{eq:NewCond1}, we have that there exist two fixed positive constants, $\bar c_1,\bar c_2$ such that 
\begin{align}
&|g(s)|\le \bar c_1 |s|^{\g-1}, & \hbox{for all }|s|\le \bar c_2, \label{gbarc}
\\
&|G(s)|\le \bar c_1 |s|^{\g}, & \hbox{for all }|s|\le \bar c_2,
\label{Gbarc}
\\
&|g_1(s)|\le \bar c_1 |s|^{\g-1}, & \hbox{for all }|s|\le \bar c_2, \label{g1barc}
\\
&|G_1(s)|\le \bar c_1 |s|^{\g}, & \hbox{for all }|s|\le \bar c_2.
\label{G1barc}
\end{align}

\begin{Lem}\label{bendef}
For any $u\in \X$, $\irn G(u)\, dx$ and $\irn g(u)u \, dx$ are well defined. The same is true for $\irn G_i(u)\, dx$ and $\irn g_i(u)u \, dx$, for $1=1,2$.
\end{Lem}

\begin{proof}
Let $u\in \X$. Since $\X$ is embedded into $L^\g(\RN)\cap L^\infty(\RN)$, we have that 
\begin{align*}
\irn |G(u)|\, dx &=\int_{\{|u|\le \bar c_2\}}|G(u)|\, dx 
+\int_{\{|u|> \bar c_2\}} |G(u)|\, dx
\\
&\le \bar c_1\int_{\{|u|\le \bar c_2\}}|u|^{\g}\, dx 
+{\rm meas}\{|u|> \bar c_2\}\cdot \max_{\{s\le \|u\|_\infty\}}|G(s)|
\\
&\le \bar c_1|u|_{\g}^{\g}
+{\rm meas}\{|u|> \bar c_2\}\cdot \max_{\{s\le \|u\|_\infty\}}|G(s)|<+\infty.
\end{align*}
The arguments are similar for $\irn g(u)u \, dx$, $\irn G_i(u)\, dx$ and $\irn g_i(u)u \, dx$, $1=1,2$.
\end{proof}

\begin{Lem}\label{le:convG}
If $u_n\weakto u_0$ in $\X$, then 
\begin{equation}\label{conv-g1}
\lim_n\int_{\R^N}g_1(u_n)u_n\, dx= \int_{\R^N}g_1(u_0)u_0\, dx
\end{equation}
and 
\begin{equation}\label{conv-g2}
\lim_n \int_{\R^N}G_1(u_n)\, dx= \int_{\R^N}G_1(u_0)\, dx.
\end{equation}
\end{Lem}

\begin{proof}
Here we follow some ideas of \cite[Corollary 3.6]{MederskiBL} (cf. \cite{MederskiBL2}) and we divide the proof into two intermediate steps by which the conclusion follows immediately.
\\
{\sc Step 1}: We claim that
\begin{equation}\label{step1}
\lim_{n} \int_{\R^N}g_1(u_n)(u_n-u_0)\,dx
=0.
\end{equation}
Since $\{u_n\}$ is bounded in $\X$ then, by the continuous embedding of $\X$ into $L^\infty(\RN)$, we infer that there exists $M>0$ such that $|u_n|_\infty\le M$, for any $n\ge 1$. 
Take any $\eps>0$ and $\beta>2^*$. Then, by \eqref{eq:NewCond1}, we find $0<\delta<M$ and $c_\eps>0$ such that 
\begin{align*}
&|g_1(s)|\leq \eps |s|^{\g-1}
\quad\hbox{ if }|s|\in [0,\delta],
%\\
%&|g_1(s)|\leq \eps \Big(e^{\alpha |s|^{\gamma_2}}-1\Big)
%\quad\hbox{ if }|s|>M,
\\
&|g_1(s)|\leq c_\eps |s|^{\beta-1}
\quad\hbox{ if }|s|\in (\delta,M].
\end{align*}
Therefore
\begin{align*}
\int_{\R^N}|g_1(u_n)(u_n-u_0)|\, dx
&\leq \eps \int_{\R^N}|u_n|^{\g-1}|u_n-u_0| \,dx
%+\eps \int_{\R^N}\Big(e^{\alpha|u_n|^{\gamma_2}}-1\Big)|u_n-u_0|\, dx
%\\
%&\quad
+c_\eps \int_{\R^N}|u_n|^{\beta-1}|u_n-u_0|\, dx,
\end{align*}
and, by the compact embedding of $\X$ into $L^\beta(\RN)$, the boundedness of the sequence $\{u_n\}$ in $\X$, we infer that
$$\limsup_n \int_{\R^N}|g_1(u_n)(u_n-u_0)|\, dx\leq \eps C$$
for some constant $C>0$ and so \eqref{step1} is proved.
\\
{\sc Step 2}: We claim that
\begin{equation*}\label{step2}
\lim_{n} \int_{\R^N}g_1(u_n)u_0\,dx
=\int_{\R^N}g_1(u_0)u_0\,dx.
\end{equation*}
Since the sequence $\{g_1(u_n)u_0\}$ is uniformly integrable and tight, then the conclusion follows by Vitali Convergence Theorem.
\\
{\sc Step 3}: We claim that
\begin{equation*}\label{step3}
\lim_{n}\left(\int_{\R^N}g_1(u_n)u_n\, dx-\int_{\R^N}g_1(u_n)(u_n-u_0)\, dx\right)= \int_{\R^N}g_1(u_0)u_0\, dx.
\end{equation*}
Indeed, if we set $\phi_n(s)=g_1(u_n)(u_n-su_0)$, for any $n\in \N$ and $s\in [0,1]$,  taking in account Step 2,  we have
\begin{align*}
&\lim_{n}\left(\int_{\R^N}g_1(u_n)u_n\, dx-\int_{\R^N}g_1(u_n)(u_n-u_0)\, dx\right)
\\
&\qquad=\lim_{n}\irn \big(\phi_n(0)-\phi_n(1)\big) dx
=-\lim_{n}\irn \left(\int_0^1\phi_n'(s)\, ds\right)dx
\\
&\qquad =\int_0^1\left(\lim_{n}\irn g_1(u_n)u_0\, dx\right)ds
=\int_0^1\left(\irn g_1(u_0)u_0\, dx\right)ds
=-\int_0^1\left(\irn\phi_0'(s)\, dx\right)ds
\\
&\qquad=\irn \big(\phi_0(0)-\phi_0(1)\big) dx
=\irn g_1(u_0)u_0\, dx. 
\end{align*} 
The proof of \eqref{conv-g2} is similar.
\end{proof}

%
%
%{\color{red}
%\begin{proof}
%	Take any $\eps>0$ and $\beta>2^*$. Then we find $0<\delta<M$ and $c_\eps>0$ such that 
%	\begin{align*}
%		G_1(s)&\leq \eps |s|^{\gamma_1}\quad\hbox{ if }|s|\in [0,\delta],\\
%		%		\Psi(s)&\leq& c s^{p}\quad\hspace{1.5mm}\hbox{ for }s\in (\delta ,M],\\
%		G_1(s)&\leq \eps \Big(e^{\alpha |s|^{\gamma_2}}-1\Big)\quad\hbox{ if }|s|>M,\\
%		G_1(s)&\leq c_\eps |s|^{\beta}\quad\hbox{ if }|s|\in (\delta,M].
%	\end{align*}
%Then 
%$$\int_{\R^N}G_1(u_n)\, dx\leq \eps \int_{\R^N}\left(|u_n|^{\gamma_1}+\Big(e^{\alpha |u_n|^{\gamma_2}}-1\Big)\right)\, dx+c_\eps \int_{\R^N}|u_n|^{\beta}\, dx$$
%and if $u_0=0$, then passing to a subsequence 
%$$\int_{\R^N}G_1(u_n)\, dx\leq \eps C$$
%for some constant $C>0$ and we  infer that $\int_{\R^N}G_1(u_n)\, dx\to 0$.
%If $u_0\neq 0$, then by Vitali's convergence theorem as in \cite[Corollary 3.6]{MederskiBL} we infer that 
%$$\int_{\R^N}G_1(u_n)\, dx-\int_{\R^N}G_1(u_n-u_0)\, dx\to \int_{\R^N}G_1(u_0)\, dx$$
%\todo[inline]{details...}
%and we conclude.
%\end{proof}
%}

Solutions of \eqref{eqTheta} will be found as critical points of the functional $I_\t:\X\to \R$ defined as 
\[
I_\t(u)=\frac12 \irn A_\t(|\n u|^2)\, dx+\irn G_2(u)\, dx-\irn G_1(u)\, dx.
\]
The functional is well defined in $\X$ by \eqref{A2q}.

\begin{Lem} \label{PM1} 
For any $\t\in (0,\t_1]$, the functional $I_\t:\X \to \mathbb{R}$ verifies the mountain pass geometry. More precisely:
\begin{itemize}
\item[(i)] there are $\alpha, \rho>0$ such that
$I_\t(u) \geq \alpha$, for $\|u\|=\rho$;	
\item[(ii)] there is $\bar u \in \X \setminus\{0\}$, independent of $\t\in (0,\t_1]$, with $\|\bar u \|>\rho$ and $|\n \bar u|<1-\t_1$, almost everywhere in $\RN$, and such that $I_\t(\bar u)<0$.
\end{itemize}
\end{Lem}

\begin{proof}
(i) 
By the continuous embedding of $\X$ into $L^\infty(\RN)$, and by \eqref{eq:NewCond1}, we can consider $\rho>0$ sufficiently small such that
\begin{equation*}
G_1(u(x))\le \frac m{2\g} |u(x)|^{\g}, \qquad \hbox{a.e. $x\in \RN$ and for any $u\in \X$ with $\|u\|=\rho$.}
\end{equation*} 
Hence, by \eqref{eq:G2} and \eqref{A2q}, for any $u\in \X$ with $\|u\|=\rho$, we have
    \begin{align*}
        I_\t(u) & \ge \frac{\bar c} 2 \left(|\n u|^2_2 + |\n u|^q_q\right)
        +  \frac m{2\g} |u|^{\g}_{\g}
\ge c\|u\|^\b\ge \a>0,
    \end{align*}
where $\b=\max\{2,q,\g\}$.
\\
(ii) Let $u_R\in \X$ such that, for any $x\in \RN$,
\[
u_R(x):=
\begin{cases}
\xi_0 & \hbox{in }B_R,
\\
-\frac{\xi_0}{\sqrt{R}}|x| +\xi_0 (1+\sqrt{R})& \hbox{in }B_{R+\sqrt{R}}\setminus B_R,
\\
0& \hbox{in }\RN \setminus B_{R+\sqrt{R}}.
\end{cases}
\]
Arguing as in \cite{BerLions}, for $R$ sufficiently large, we have $\irn G(u_R)\, dx>0$ and, clearly,  $|\n u_R|<1-\t_1$. Moreover, for any $t>1$, we have also that $|\n u_R(\cdot/t)|\le 1-\t_1$ and so, denoting $\bar u =u_R(\cdot/t)$, with $R$ and $t$ sufficiently large and independently by $\t\in (0,\t_1]$, we have $\|\bar u\|> \rho$ and 
\[
I_\t(\bar u)\le c_1\left(t^{N-2}|\n u_R|^2_2+t^{N-q}|\n u_R|^q_q\right)
-t^N \irn G(u_R)\, dx<0.
\]
\end{proof}

Let us define the mountain pass level for the functional $I_\t$
\[
m_\t:=\inf_{\g\in \Gamma}\max_{t\in [0,1]}I_\t(\g(t)), 
\]
where 
\[
\Gamma:=\{\g\in \cC([0,1],\X)\mid \g(0)=0, \g(1)=\bar u\}.
\]
By Lemma \ref{PM1}, we deduce that $m_\t\ge \a$, for any $\t\in (0,\t_1]$.

Observe that, since $|\n \bar u|<1-\t_1$, we have that $I_{\t_1}(t \bar u)=I_\t(t\bar u)$, for any $t\in [0,1]$ and for any $\t\in (0,\t_1]$. Hence we deduce that
\begin{equation*}\label{mtheta}
m_\t\le \max_{t \in [0,1]}I_\t(t\bar u)
=\max_{t \in [0,1]}I_{\t_1}(t\bar u),
\end{equation*}
for any $\t\in (0,\t_1]$. Hence there exists $c>0$ (independent of  $\t\in (0,\t_1]$) such that 
\begin{equation}\label{mtbdd}
0< m_\t \le c, \qquad \hbox{for any $\t\in (0,\t_1]$.}
\end{equation}

Following \cite{HIT,jj}, we define the  functional $J_\t:\R\times \X\to \R$  as 
\[
J_\t(\s, u)=
I_\t(u(e^{-\s}\cdot ))
=\frac{e^{N\s}}2 \irn A_\t(e^{-2\s}|\n u|^2)\, dx+e^{N\s}\irn G_2(u)\, dx-e^{N\s}\irn G_1(u)\, dx.
\]
With similar arguments of Lemma \ref{PM1}, also $J_\t$ has a mountain pass geometry and we can define its mountain pass level as
\[
\tilde m_\t:=\inf_{(\s,\g)\in \Sigma\times \Gamma}\max_{t\in [0,1]}J_\t\big(\s(t),\g(t)\big), 
\]
where 
\[
\Sigma:=\{\s\in \cC([0,1],\R)\mid \s(0)=\s(1)=0\}.
\]
Observe that arguing as in \cite[Lemma 3.1]{HIT}, 
we obtain
\begin{Lem}
For any $\t\in (0,\t_1]$, the mountain pass levels of $I_\t$ and $J_\t$ coincide, namely $m_\t=\tilde m_\t$.
\end{Lem}

Now, as an immediate consequence of Ekeland's variational principle \cite[Theorem 2.8]{Willem} (cf.  \cite[Lemma 2.3]{jj}) we obtain the following results.
\begin{Lem}\label{le:ekeland}
Let $\t\in (0,\t_1]$ and $\eps>0$. Suppose that $\tilde\gamma\in \Sigma \times\Gamma$ satisfies 
\[
\max_{t \in [0,1]}J_\t(\tilde \gamma(t))\le m_\t+\eps,
\]
then there exists $(\s, u)\in \R\times \X$ such that
\begin{enumerate}
\item ${\rm dist}_{\R \times \X}\big((\t,u),\tilde{\gamma}([0,1])\big)\le 2 \sqrt{\eps}$;
\item $J_\t(\s,u)\in [m_\t-\eps,m_\t+\eps]$;
\item $\|D J_\t(\s,u)\|_{\R \times \X^*}\le 2 \sqrt{\eps}$.
\end{enumerate}
\end{Lem}

\begin{Prop}\label{pr:sequence}
For any $\t\in (0,\t_1]$, there exists a sequence $\{(\s_n,u_n)\} \subset \R \times \X$ such that, as $n \to +\infty$, we get 
\begin{enumerate}
\item $\s_n \to 0$;
\item $J_\t(\s_n,u_n)\to m_\t$; 
\item $\de_\s J_\t(\s_n,u_n)\to 0$; 
\item $\de_u J_\t(\s_n,u_n)\to 0$ strongly in $\X^*$. 
\end{enumerate}

\end{Prop}
\begin{proof}
In view of Lemma \ref{le:ekeland} we conclude by letting $\eps\to0$.
\end{proof}

Now we find a radial solution of the truncated problem \eqref{eqTheta}.

\begin{Prop}\label{pr:sol}
For any $\t\in (0,\t_1]$, there exists $u_\t\in \X$ a non-trivial  solution of \eqref{eqTheta} such $I_\t(u_\t)=m_\t$. Moreover
there exists $C>0$ such that 
\begin{equation}\label{unifbdd}
\|u_\t\|_0\le C, \quad\hbox{ for any }\t\in (0,\t_1].
\end{equation}
Finally $u_\t$ is a weak solution of 
\begin{equation}\label{eqrad}
-\big(r^{N-1}a_\t(|u'_\t(r)|^2)u'_\t(r)\big)'=r^{N-1}g(u_\t(r)),
\end{equation}
namely 
\[
\int_0^{+\infty}r^{N-1}a_\t(|u'_\t(r)|^2)u'_\t(r)v'(r)\, dr
=\int_0^{+\infty}r^{N-1}g(u_\t(r))v(r)\,dr,
\]
for all $v\in \X$.
\end{Prop}

\begin{proof}
Fix $\t\in (0,\t_1]$. By Proposition \ref{pr:sequence}, there exists a sequence $\{(\s_n,u_n)\} \subset \R \times \X$ such that
\begin{equation}\label{sistema}
\begin{cases}
\dis\frac{e^{N\s_n}}{2}\irn A_\t(e^{-2\s_n}|\n u_n|^2)\, dx
+e^{N\s_n}\irn G_2(u_n)\, dx
-e^{N\s_n}\irn G_1(u_n) \, dx =m_\t+o_n(1),
\\[7mm]
\dis\frac{Ne^{N\s_n}}{2}\irn A_\t(e^{-2\s_n}|\n u_n|^2)\, dx
-e^{(N-2)\s_n}\irn a_\t(e^{-2\s_n}|\n u_n|^2)|\n u_n|^2\, dx
\\[2mm]
\dis \hspace{5.5cm}+Ne^{N\s_n}\irn G_2(u_n)\, dx
-Ne^{N\s_n}\irn G_1(u_n) \, dx =o_n(1),
\\[7mm]
\dis e^{(N-2)\s_n}\irn a_\t(e^{-2\s_n}|\n u_n|^2)|\n u_n|^2\, dx
+e^{N\s_n}\irn g_2(u_n)u_n\, dx
-e^{N\s_n}\irn g_1(u_n) u_n\, dx =o_n(1)\|u_n\|.
\end{cases}
\end{equation}
From the first and the second equation of the previous system we get
\[
e^{(N-2)\s_n}\irn a_\t(e^{-2\s_n}|\n u_n|^2)|\n u_n|^2\, dx=N m_\t+o_n(1).
\]
Therefore, since $\s_n \to 0$, as $n \to +\infty$, by \eqref{a2q} we deduce that 
$\{u_n\}$ is a bounded sequence in $\X_0$ and so also in $L^\infty(\RN)$, namely there exists $\bar C>0$ such that $|u_n|_\infty\le \bar C$, for any $n\ge 1$. 
This implies that, by \eqref{eq:NewCond1} and  Lemma \ref{le:strauss}, there exists $R>1$ such that
\begin{equation*}
G_1(u_n(x))\le \frac m{2\g} |u_n(x)|^{\g}, \qquad \hbox{a.e. $x\in \RN$ with $|x|\ge R$ and for any $n\ge 1$.}
\end{equation*} 
Hence
\begin{align*}
\irn G_1(u_n)\, dx= \int_{B_R} G_1(u_n)\, dx
+\int_{B_R^c} G_1(u_n)\, dx
\le C\max_{\{s\le \bar C\}}|G_1(s)|+\frac m{2\g}\irn  |u_n(x)|^{\g}\, dx.
\end{align*}
By this, by \eqref{eq:G2} and by the first equation of \eqref{sistema},  we infer that $\{u_n\}$ is a bounded sequence also in $\X$. Then there exists $u_\t\in \X$ such that $u_n \weakto u_\t$ in $\X$. Since $\de_u J_\t(\s_n,u_n)\to 0$ strongly in $\X^*$ and $\s_n \to 0$, we have that $u_\t$ is a weak (possibly trivial) solution of \eqref{eqTheta} and so it satisfies
\[
\irn a_\t(|\n u_\t|^2)|\n u_\t|^2\, dx
+\irn g_2(u_\t)u_\t\, dx
=\irn g_1(u_\t) u_\t\, dx.
\]
Since  $u_n \weakto u_\t$ in $\X$, by the weak lower semicontinuity and the Fatou's Lemma we have that
\begin{align*}
\irn a_\t(|\n u_\t|^2)|\n u_\t|^2\, dx 
&\le \liminf_{n \to +\infty}\irn a_\t(|\n u_n|^2)|\n u_n|^2\, dx,
\\
\irn g_2(u_\t)u_\t\, dx
&\le \liminf_{n \to +\infty}\irn g_2(u_n)u_n\, dx;
\end{align*}
while, by Lemma \ref{le:convG}, we have
\[
\irn g_1(u_\t)u_\t\, dx
= \lim_{n \to +\infty}\irn g_1(u_n)u_n\, dx.
\]
Therefore, by the third equation of \eqref{sistema},
\begin{align*}
&\irn a_\t(|\n u_\t|^2)|\n u_\t|^2\, dx 
+\irn g_2(u_\t)u_\t\, dx
\\
&\qquad\le \liminf_{n\to +\infty}\left[\irn a_\t(|\n u_n|^2)|\n u_n|^2\, dx
+\irn g_2(u_n)u_n\, dx\right]
\\
&\qquad=\liminf_{n\to +\infty}\left[e^{(N-2)\s_n}\irn a_\t(e^{-2\s_n}|\n u_n|^2)|\n u_n|^2\, dx
+e^{N\s_n}\irn g_2(u_n)u_n\, dx\right]
\\
&\qquad= \dis\liminf_{n\to +\infty} \left[e^{N\s_n}\irn g_1(u_n) u_n\, dx +o_n(1)\|u_n\|\right]
\\
&\qquad=\irn g_1(u_\t) u_\t\, dx
=\irn a_\t(|\n u_\t|^2)|\n u_\t|^2\, dx
+\irn g_2(u_\t)u_\t\, dx
\end{align*}
and so
\begin{align}
\irn a_\t(|\n u_\t|^2)|\n u_\t|^2\, dx 
&= \lim_{n \to +\infty}\irn a_\t(|\n u_n|^2)|\n u_n|^2\, dx, \label{conv1}
\\
\irn g_2(u_\t)u_\t\, dx
&= \lim_{n \to +\infty}\irn g_2(u_n)u_n\, dx. \label{conv2}
\end{align}
In view of Lemma \ref{lem:convgrad} equation \eqref{conv1} implies that $u_n \to u_\t$ strongly in $\X_0$.
\\
Moreover, since, by \eqref{eq:NewCond2}, we know that for any $s\in \R$ we can write $g_2(s)s=m|s|^{\g}+h(s)$, where $h$ is a non-negative continuous function, 
by Fatou's Lemma we deduce that
\begin{align*}
\irn |u_\t|^{\g}\, dx 
&\le \liminf_{n \to +\infty}\irn | u_n|^{\g}\, dx,
\\
\irn h(u_\t)\, dx
&\le \liminf_{n \to +\infty}\irn h(u_n)\, dx.
\end{align*}
These last two inequalities and \eqref{conv2} imply that
\[
\irn |u_\t|^{\g}\, dx 
=\lim_{n \to +\infty}\irn | u_n|^{\g}\, dx
\]
and so, actually, $u_n \to u_\t$ strongly in $\X$ and so $I_\t(u_\t)=m_\t$. 
\\
Finally, since 
\[
\irn a_\t(|\n u_\t|^2)|\n u_\t|^2\, dx =Nm_\t,
\]
by \eqref{mtbdd} and \eqref{a2q}, we prove that there exists $C>0$ such that $\|u_\t\|_0\le C$, for any $\t\in (0,\t_1]$.

\end{proof}

We are now able to conclude the proof of our main theorem in the positive mass case.

\begin{proof}[Proof of Theorem \ref{main}]
By Proposition \ref{pr:sol}, for any $\t\in (0,\t_1]$, there exists $u_\t\in \X$ a nontrivial solution of \eqref{eqTheta} such $I_\t(u_\t)=m_\t$. 
Since $q>N$, by \cite{L}, we deduce that $u_\t\in \cC^{1,\a}$, for some $\a\in (0,1)$. 
\\
Let us prove the following

\medspace \noindent
{\sc Claim:} there exists $C>0$ such that 
\begin{equation}\label{claim}
|a_\t(|u'_\t(r)|^2)u'_\t(r)|\le C, \qquad \hbox{for any $r\ge 0$ and $\t\in (0,\t_1]$}.
\end{equation}

\medspace

\noindent By the regularity of $u_\t$, we infer that $u_\t'(0)=0$ and so also 
\[
a_\t(|u'_\t(0)|^2)u'_\t(0)=0.
\]
We now consider the case $r>0$. 
Integrating the equation \eqref{eqrad}, for any $r>0$, we have
\begin{equation*}
-a_\t(|u'_\t(r)|^2)u'_\t(r)
=\frac1{r^{N-1}}\int_0^r s^{N-1}g(u_\t(s))\, ds.
\end{equation*}
By Lemma \ref{le:strauss} and by \eqref{unifbdd}, we deduce that there exists $R>1$, such that 
\begin{equation}\label{stimaunif}
 |u_\t(r)|\le \bar c_2, \quad\hbox{for any  $\t\in (0,\t_1]$ and for any $r>R$,} 
\end{equation}
where $\bar c_2$ is  defined in \eqref{gbarc}. 
\\
By the continuous embedding of $\X_0$ in $L^\infty(\RN)$ and \eqref{unifbdd},  there exists $C>0$ such that $|u_\t|_\infty\le C\|u_\t\|_0\le C$, for any $\t\in (0,\t_1]$, and so we have that, for any $0<r\le R$ and $\t\in (0,\t_1]$,  
\begin{equation*}
|a_\t(|u'_\t(r)|^2)u'_\t(r)|
\le \frac1{r^{N-1}}\int_0^r s^{N-1}|g(u_\t(s))|\, ds\le C.
\end{equation*}
While, for any $r> R$, 
\begin{align*}
|a_\t(|u'_\t(r)|^2)u'_\t(r)|
&\le \frac1{r^{N-1}}\int_0^r s^{N-1}|g(u_\t(s))|\, ds
\\
&\le \frac1{r^{N-1}}\left(\int_0^R s^{N-1}|g(u_\t(s))|\, ds
+\int_R^r s^{N-1}|g(u_\t(s))|\, ds\right)\\
&\le \frac C{r^{N-1}}
+\underbrace{\frac{c_1}{r^{N-1}}\int_1^r s^{N-1}|g(u_\t(s))|\, ds}_{(A)}.
\end{align*}
We have to estimate $(A)$. First of all, by Lemma \ref{le:strauss} and \eqref{unifbdd}, for $r>1$, we have that
\[
|u_\t(r)| \le C r^{-\frac{N-2}2} |\n u_\t|_2\le \bar C r^{-\frac{N-2}2}.
\]
Hence, by \eqref{stimaunif} and \eqref{gbarc}, since $\g\ge 2^*/2$,
\begin{align*}
(A)&
\le \frac{C}{r^{N-1}}\int_1^r s^{N-1}|u_\t(s)|^{\g-1}\, ds
\le \frac{ C}{r^{N-1}}\int_1^r s^{N-1-\frac{N-2}2(\g-1)}\, ds
\le C \left(r^{1-\frac{N-2}2(\g-1)}+1\right)\le C.
\end{align*}
Therefore the claim  is proved.
\\
Now we conclude if we show
the existence of  $\bar \t\in (0,\t_1]$ such that 
\begin{equation}\label{claim2}
|u'_{\bar \t}(r)|\le 1-\bar \t, \qquad \hbox{for any $r\ge 0$}.
\end{equation}
Suppose by contradiction that \eqref{claim2} does not hold, then there exists a sequence $\{\t_n\}\subset (0,\t_1]$ which tends to zero and a sequence $\{r_n\}\subset  \R_+$ such that 
\[
\lim_n |u'_{\t_n}(r_n)|=1,
\]
which implies that (by \eqref{a1})
\[
\lim_n a_{\t_n}(|u'_{\t_n}(r_n)|)|u'_{\t_n}(r_n)|
=+\infty.
\]
Thus we obtain a contradiction with \eqref{claim}.
\\
Finally, taking into account \eqref{a2q}, \eqref{A2q} and Lemma \ref{bendef}, we get
\[
\irn A(|\n u_{\bar \t}|^2)\,dx, \irn a(|\n u_{\bar \t}|^2)|\n u_{\bar \t}|^2\,dx, \irn |G(u_{\bar \t})|\, dx<+\infty.
\]
\end{proof}

\begin{Rem}
	Note that in \eqref{g1}  we assume $\gamma \geq 2^*/2$. Actually this is a  technical requirement which is essential only in proof of Theorem \ref{main}, while, in all the other our arguments, we can simply take $\gamma>1$. Therefore there exists a radial solution to
	\eqref{eqTheta} for any $\gamma>1$. 
\end{Rem}

\section{The zero mass case}\label{se0}

In this section we deal with the zero mass case, namely, we will assume that $g$ satisfies \eqref{g0} and \eqref{g3}. Moreover  $\gamma>N$ and \eqref{g1z} holds, or    $\gamma\leq N$ and  both \eqref{g1z}, \eqref{g1zinf} hold.

Let 
$g_1(s):=
\max\{g(s),0\}$ 
and $g_2(s):=g_1(s)-g(s)$ for $s\geq 0$ and then we can extend them as odd functions for $s<0$. 
%{\color{blue}\[
%g_1(s):=
%\begin{cases}
%\max\{g(s),0\} &\hbox{for }s\geq 0,
%\\
%\min\{g(s),0\} &\hbox{for }s< 0.
%\end{cases}
%\]
%and $g_2(s):=g_1(s)-g(s)$, for $s\in \R$. }
%\todo[inline]{as in the positive mass case, the previous definitions were not clear}
Then $g_1(s),g_2(s)\geq 0$, for $s\geq 0$ and
\begin{align}\label{eq:NewCond1z}
\lim_{s\to 0} g_1(s)/|s|^{\gamma-1}&=0, \quad
\text{for some }\g>2^*.
\end{align}
Moreover, whenever $\g\in (2^*,N]$, we have
\begin{equation}\label{eq:NewCond1zinf}
\lim_{s\to+\infty} g_1(s)/|s|^{q^*-1}=0, 
\end{equation}
For $i=1,2$ we set 
$$G_i(s)=\int_0^s g_i(t)\, dt$$ 
and note that $G_i(s) \ge 0$ for $s\in\R$.

In view of  \eqref{g1z}, there exist two positive constants, $\bar c_1$ and $\bar c_2$, such that 
\begin{align}
&|g(s)|\le \bar c_1 |s|^{\g-1}, & \hbox{for all }|s|\le \bar c_2, \label{gbarcz}
\\
&|G(s)|\le \bar c_1 |s|^{\g}, & \hbox{for all }|s|\le \bar c_2,
\label{Gbarcz}
\\
&|g_1(s)|\le \bar c_1 |s|^{\g-1}, & \hbox{for all }|s|\le \bar c_2, \label{g1barcz}
\\
&|G_1(s)|\le \bar c_1 |s|^{\g}, & \hbox{for all }|s|\le \bar c_2.
\label{G1barcz}
\end{align}
Moreover, in the case $\g\in (2^*,N]$, by  \eqref{g1z} and \eqref{g1zinf}, there exists a positive constant $\bar c_3$ such that 
\begin{align}
&|g(s)|\le \bar c_3 \left(|s|^{\g-1}+|s|^{q^*-1}\right), & \hbox{for all }s\in \R, \label{gbarczinf}
\\
&|G(s)|\le \bar c_3 \left(|s|^{\g}+|s|^{q^*}\right), & \hbox{for all }s\in \R,
\label{Gbarczinf}
\\
&|g_1(s)|\le \bar c_3 \left(|s|^{\g-1}+|s|^{q^*-1}\right), & \hbox{for all }s\in \R, \label{g1barczinf}
\\
&|G_1(s)|\le \bar c_3 \left(|s|^{\g}+|s|^{q^*}\right), & \hbox{for all }s\in \R.
\label{G1barczinf}
\end{align}

Arguing as in the proof of Lemma \ref{bendef}, we have
\begin{Lem}
For any $u\in \X_0$, $\irn G(u)\, dx$ and $\irn g(u)u \, dx$ are well defined. The same is true for $\irn G_i(u)\, dx$ and $\irn g_i(u)u \, dx$, for $1=1,2$.
\end{Lem}

%\begin{proof}
%Let $u\in \X$. Since $\X$ is embedded into $L^\infty(\RN)$, we have that 
%\begin{align*}
%\irn |G(u)|\, dx &=\int_{\{|u|\le \bar c_2\}}|G(u)|\, dx 
%+\int_{\{|u|> \bar c_2\}} |G(u)|\, dx
%\\
%&\le \bar c_1\int_{\{|u|\le \bar c_2\}}|u|^{\g}\, dx 
%+{\rm meas}\{|u|> \bar c_2\} \max_{\{s\le \|u\|_\infty\}}|G(s)|
%\\
%&\le \bar c_1|u|_{\g}^{\g}
%+{\rm meas}\{|u|> \bar c_2\} \max_{\{s\le \|u\|_\infty\}}|G(s)|<+\infty.
%\end{align*}
%The arguments for the second integral are similar.
%\end{proof}

The following compactness results hold.

\begin{Lem}\label{le:convGz}
If $u_n\weakto u_0$ in $\X_0$, then 
\begin{equation*}\label{conv-g1z}
\lim_n\int_{\R^N}g_1(u_n)u_n\, dx= \int_{\R^N}g_1(u_0)u_0\, dx
\end{equation*}
and 
\begin{equation*}\label{conv-g2z}
\lim_n \int_{\R^N}G_1(u_n)\, dx= \int_{\R^N}G_1(u_0)\, dx.
\end{equation*}
\end{Lem}

\begin{proof}
In the case $\g>N$, the arguments are similar to those of the proof of Lemma \ref{le:convG}. Here we treat only the case $\g\in (2^*,N]$, enlightening only the main differences.
\\
By \eqref{eq:NewCond1z} and \eqref{eq:NewCond1zinf}, take any $\eps>0$ and $\beta\in(2^*,q^*)$, then we find $\delta>0$ and $c_\eps>0$ such that 
\begin{align*}
&|g_1(s)|\leq \eps |s|^{\g-1}
\quad\hbox{ if }|s|\in [0,\delta],
\\
&|g_1(s)|\leq c_\eps |s|^{\beta-1}
\quad\hbox{ if }|s|\in (\delta,1/\delta),
\\
&|g_1(s)|\leq \eps |s|^{q^*-1}
\quad\hbox{ if }|s|\in [1/\delta, +\infty).
\end{align*}
Therefore
\begin{align*}
\int_{\R^N}|g_1(u_n)(u_n-u_0)|\, dx
&\leq \eps \int_{\R^N}|u_n|^{\g-1}|u_n-u_0| \,dx
%+\eps \int_{\R^N}\Big(e^{\alpha|u_n|^{\gamma_2}}-1\Big)|u_n-u_0|\, dx
%\\
%&\quad
+c_\eps \int_{\R^N}|u_n|^{\beta-1}|u_n-u_0|\, dx
\\
&\quad+\eps \int_{\R^N}|u_n|^{q^*-1}|u_n-u_0| \,dx,
\end{align*}
and, by the compact embedding of $\X_0$ into $L^\beta(\RN)$, the boundedness of the sequence $\{u_n\}$ in $\X_0$, we infer that
$$\limsup_n \int_{\R^N}|g_1(u_n)(u_n-u_0)|\, dx\leq \eps C$$
for some constant $C>0$. Now the proof goes on similarly as in Lemma \ref{le:convG}.
\end{proof}

Solutions of \eqref{eqThetaz} will be found as critical points of the functional $I_\t:\X_0\to \R$ defined as 
\[
I_\t(u)=\frac12 \irn A_\t(|\n u|^2)\, dx+\irn G_2(u)\, dx-\irn G_1(u)\, dx.
\]
which is well defined in $\X_0$. 
Here and in what follows, with an abuse of notation, we use  $I_\theta$, $J_\theta$, $m_\theta$, $\tilde{m}_\theta$, $\Gamma$, and  $\Sigma$ in the zero mass setting, as well.

We show that  $I_\t$ satisfies the mountain pass geometry.

\begin{Lem} \label{PM1z} 
For any $\t\in (0,\t_1]$, the functional $I_\t:\X_0 \to \mathbb{R}$ verifies the mountain pass geometry. More precisely:
\begin{itemize}
\item[(i)] there are $\alpha, \rho>0$ such that
$I_\t(u) \geq \alpha$, for $\|u\|_0=\rho$;	
\item[(ii)] there is $\bar u \in \X_0 \setminus\{0\}$, independent of $\t\in (0,\t_1]$, with $\|\bar u \|_0>\rho$ and $|\n \bar u|<1-\t_1$, almost everywhere in $\RN$, and such that $I_\t(\bar u)<0$.
\end{itemize}
\end{Lem}

\begin{proof}
(i) 
We start with the case $\g>N$. We fix $q\in (N,\g)$. By the continuous embedding of $\X_0$ into $L^\infty(\RN)$, and by \eqref{Gbarcz}, we can consider $\rho>0$ sufficiently small such that
\begin{equation*}
G(u(x))\le \bar c_1|u(x)|^{\g}, \qquad \hbox{a.e. $x\in \RN$ and for any $u\in \X_0$ with $\|u\|_0=\rho$.}
\end{equation*} 
Hence, by \eqref{A2q} and since $\X_0$ is embedded into $L^\g(\RN)$, for any $u\in \X_0$ with $\|u\|_0=\rho$, we have
\begin{align*}
I_\t(u) & \ge c \left(|\n u|^2_2 + |\n u|^q_q-|u|_\g^\g\right) 
\ge c \left(|\n u|^2_2 + |\n u|^q_q
-|\n u|^\g_2 - |\n u|^\g_q\right)
\ge  \a>0.
\end{align*}
Let us consider now the case $\g\in (2^*,N]$.
By \eqref{eq:NewCond1z} and \eqref{eq:NewCond1zinf}, take any $\eps>0$ and $\beta\in(\max\{2^*,q\},q^*)$, then we find $c_\eps>0$ such that 
%\begin{align*}
%&|g_1(s)|\leq \eps |s|^{\g-1}
%\quad\hbox{ if }|s|\in [0,\delta],
%\\
%&|g_1(s)|\leq c_\eps |s|^{\beta-1}
%\quad\hbox{ if }|s|\in (\delta,1/\delta),
%\\
%&|g_1(s)|\leq \eps |s|^{q^*-1}
%\quad\hbox{ if }|s|\in [1/\delta, +\infty).
%\end{align*}
\[
0\le G_1(s)\le \eps\left(|s|^\g+|s|^{q^*}\right)+c_\eps |s|^\beta, \qquad\text{for all }s\in \R.
\]
Hence, if $\rho<1$, we have
\begin{align*}
I_\t(u) & \ge c \left(|\n u|^2_2 + |\n u|^q_q\right)
-\eps \left(|u|^\g_\g+|u|^{q^*}_{q^*}\right)-c_\eps |u|^\beta_\b
\\
&\ge c \left[|\n u|^2_2 + |\n u|^q_q
-\eps\left(|\n u|^\g_2 + |\n u|^\g_q
+|\n u|_2^{q^*} + |\n u|_q^{q^*}\right)
-\left(|\n u|^\b_2 + |\n u|^\b_q\right)\right]
\\
&\ge c \left[ \| u\|^q_0-\| u\|^\b_0
-\eps\left(\| u\|^\g_0 + \| u\|^{q^*}_0\right)\right]
\ge \a>0.
\end{align*}
(ii) As in the proof of Lemma \ref{PM1}.
\end{proof}

Let us define the mountain pass level for the functional $I_\t$
\[
m_\t:=\inf_{\g\in \Gamma}\max_{t\in [0,1]}I_\t(\g(t)), 
\]
where 
\[
\Gamma:=\{\g\in \cC([0,1],\X_0)\mid \g(0)=0, \g(1)=\bar u\}.
\]
By Lemma \ref{PM1}, we deduce that $m_\t\ge \a$, for any $\t\in (0,\t_1]$.

Observe that, since $|\n \bar u|<1-\t_1$, we have that $I_{\t_1}(t \bar u)=I_\t(t\bar u)$, for any $t\in [0,1]$ and for any $\t\in (0,\t_1]$. Hence we deduce that
\begin{equation*}\label{mthetaz}
m_\t\le \max_{t \in [0,1]}I_\t(t\bar u)
=\max_{t \in [0,1]}I_{\t_1}(t\bar u),
\end{equation*}
for any $\t\in (0,\t_1]$. Hence there exists $c>0$  (independent of  $\t\in (0,\t_1]$) such that 
\begin{equation}\label{mtbddz}
0< m_\t \le c_2, \qquad \hbox{for any $\t\in (0,\t_1]$.}
\end{equation}

As done in Section \ref{se>}, we define the  functional $J_\t:\R\times \X_0\to \R$ as 
\[
J_\t(\s, u)=
I_\t(u(e^{-\s}\cdot ))
=\frac{e^{N\s}}2 \irn A_\t(e^{-2\s}|\n u|^2)\, dx+e^{N\s}\irn G_2(u)\, dx-e^{N\s}\irn G_1(u)\, dx.
\]
The functional $J_\t$ has a mountain pass geometry and we can define its mountain pass level as
\[
\tilde m_\t:=\inf_{(\s,\g)\in \Sigma\times \Gamma}\max_{t\in [0,1]}J_\t\big(\s(t),\g(t)\big), 
\]
where 
\[
\Sigma:=\{\s\in \cC([0,1],\R)\mid \s(0)=\s(1)=0\}.
\]
The following  holds
\begin{Lem}
For any $\t\in (0,\t_1]$, the mountain pass levels of $I_\t$ and $J_\t$ coincide, namely $m_\t=\tilde m_\t$.
\end{Lem}

\begin{Lem}\label{le:ekelandz}
Let $\t\in (0,\t_1]$ and $\eps>0$. Suppose that $\tilde\gamma\in \Sigma \times\Gamma$ satisfies 
\[
\max_{t \in [0,1]}J_\t(\tilde \gamma(t))\le m_\t+\eps,
\]
then there exists $(\s, u)\in \R\times \X_0$ such that
\begin{enumerate}
\item ${\rm dist}_{\R \times \X_0}
\big((\t,u),\tilde{\gamma}([0,1])\big)\le 2 \sqrt{\eps}$;
\item $J_\t(\s,u)\in [m_\t-\eps,m_\t+\eps]$;
\item $\|D J_\t(\s,u)\|_{\R \times \X^*}\le 2 \sqrt{\eps}$.
\end{enumerate}
\end{Lem}

\begin{Prop}\label{pr:sequencez}
For any $\t\in (0,\t_1]$, there exists a sequence $\{(\s_n,u_n)\} \subset \R \times \X_0$ such that, as $n \to +\infty$, we get 
\begin{enumerate}
\item $\s_n \to 0$;
\item $J_\t(\s_n,u_n)\to m_\t$; 
\item $\de_\s J_\t(\s_n,u_n)\to 0$; 
\item $\de_u J_\t(\s_n,u_n)\to 0$ strongly in $\X_0^*$. 
\end{enumerate}

\end{Prop}

\begin{Prop}\label{pr:solz}
For any $\t\in (0,\t_1]$, there exists $u_\t\in \X_0$ a non-trivial  solution of \eqref{eqTheta} such $I_\t(u_\t)=m_\t$. Moreover
there exists $C>0$ such that 
\begin{equation}\label{unifbddz}
\|u_\t\|_0\le C, \quad\hbox{ for any }\t\in (0,\t_1].
\end{equation}
Finally $u_\t$ is a weak solution of 
\begin{equation}\label{eqradz}
-\big(r^{N-1}a_\t(|u'_\t(r)|^2)u'_\t(r)\big)'=r^{N-1}g(u_\t(r)),
\end{equation}
namely 
\[
\int_0^{+\infty}r^{N-1}a_\t(|u'_\t(r)|^2)u'_\t(r)v'(r)\, dr
=\int_0^{+\infty}r^{N-1}g(u_\t(r))v(r)\,dr,
\]
for all $v\in \X_0$.
\end{Prop}

\begin{proof}
Fix $\t\in (0,\t_1]$. By Proposition \ref{pr:sequencez}, there exists a sequence $\{(\s_n,u_n)\} \subset \R \times \X_0$ such that
\begin{equation*}\label{sistemaz}
\begin{cases}
\dis\frac{e^{N\s_n}}{2}\irn A_\t(e^{-2\s_n}|\n u_n|^2)\, dx
+e^{N\s_n}\irn G_2(u_n)\, dx
-e^{N\s_n}\irn G_1(u_n) \, dx =m_\t+o_n(1),
\\[7mm]
\dis\frac{Ne^{N\s_n}}{2}\irn A_\t(e^{-2\s_n}|\n u_n|^2)\, dx
-e^{(N-2)\s_n}\irn a_\t(e^{-2\s_n}|\n u_n|^2)|\n u_n|^2\, dx
\\[2mm]
\dis \hspace{5.5cm}+Ne^{N\s_n}\irn G_2(u_n)\, dx
-Ne^{N\s_n}\irn G_1(u_n) \, dx =o_n(1),
\\[7mm]
\dis e^{(N-2)\s_n}\irn a_\t(e^{-2\s_n}|\n u_n|^2)|\n u_n|^2\, dx
+e^{N\s_n}\irn g_2(u_n)u_n\, dx
-e^{N\s_n}\irn g_1(u_n) u_n\, dx =o_n(1)\|u_n\|.
\end{cases}
\end{equation*}
From the first and the second equation of the previous system we get
\[
e^{(N-2)\s_n}\irn a_\t(e^{-2\s_n}|\n u_n|^2)|\n u_n|^2\, dx=N m_\t+o_n(1).
\]
Therefore, since $\s_n \to 0$, as $n \to +\infty$, by \eqref{a2q} we deduce that 
$\{u_n\}$ is a bounded sequence in $\X_0$.
Then there exists $u_\t\in \X_0$ such that $u_n \weakto u_\t$ in $\X_0$. Since $\de_u J_\t(\s_n,u_n)\to 0$ strongly in $\X^*_0$ and $\s_n \to 0$, we have that $u_\t$ is a weak (possibly trivial) solution of \eqref{eqThetaz} and so it satisfies
\[
\irn a_\t(|\n u_\t|^2)|\n u_\t|^2\, dx
+\irn g_2(u_\t)u_\t\, dx
=\irn g_1(u_\t) u_\t\, dx.
\]
Arguing as in proof of Proposition \ref{pr:sol} we can show that
\begin{equation*}\label{conv1z}
\irn a_\t(|\n u_\t|^2)|\n u_\t|^2\, dx 
= \lim_{n \to +\infty}\irn a_\t(|\n u_n|^2)|\n u_n|^2\, dx.
\end{equation*}
In view of Lemma \ref{lem:convgrad}, we have that $u_n \to u_\t$ strongly in $\X_0$ and so $I_\t(u_\t)=m_\t$. 
\\
Finally, since 
\[
\irn a_\t(|\n u_\t|^2)|\n u_\t|^2\, dx =Nm_\t,
\]
by \eqref{mtbddz} and \eqref{a2q}, we prove that there exists $C>0$ such that $\|u_\t\|_0\le C$, for any $\t\in (0,\t_1]$.
\end{proof}

We are now able to conclude the proof of Theorem \ref{main}.

\begin{proof}[Proof of Theorem \ref{main}]
When $\g>N$ we can change slightly the arguments of Section \ref{se>}. Here we deal just with the case $2^*<\l\le N$ and so we have to assume \eqref{g1zinf}.
\\
By Proposition \ref{pr:solz}, for any $\t\in (0,\t_1]$, there exists $u_\t\in \X_0$ a nontrivial solution of \eqref{eqThetaz} such $I_\t(u_\t)=m_\t$. 
Being $q< N$, we cannot repeat the arguments of the previous section and we follow some ideas of \cite[Lemma 3.2]{BDD}. Since $u_\t$ is a solution of \eqref{eqradz} in $(0,+\infty)$, it is easy to check that $u_\t$ is regular for $r>0$. 
Moreover,  $r^{N-1}a_\t(|u'_\t(r)|^2)u'_\t(r)$ satisfies the Cauchy condition at the origin so that it has a finite limit as $r\to0$. 
We claim that
\begin{equation}\label{eq:lim0}
\lim_{r\to 0}r^{N-1}a_\t(|u'_\t(r)|^2)u'_\t(r)=0.
\end{equation}
Suppose, by contradiction, that it is different from zero and then there should exist $r_0>0$ such that $|u'_\t(r)|>1-\t$, for $r\in (0,r_0]$. Therefore, for $r$ sufficiently small, 
\[
C\le \left|r^{N-1}a_\t(|u'_\t(r)|^2)u'_\t(r)\right|
=r^{N-1}|u'_\t(r)|^{q-1},
\]
namely
\[
|u'_\t(r)|\ge Cr^{-\frac{N-1}{q-1}}.
\]
By this we have
\[
r^{N-1}a_\t(|u'_\t(r)|^2)|u'_\t(r)|^2
=r^{N-1}|u'_\t(r)|^{q}
\ge Cr^{-\frac{N-1}{q-1}}
\]
near $0$, which is not integrable since $q < N$. Since $u_\t$ is a solution of \eqref{eqradz}, we get a contradiction.
%Hence  \eqref{eq:lim0} is proved.
%\\
%On the other hand,  $u_\t$ is a solution of \eqref{eqradz} and so we have 
%\[
%\int_0^{+\infty}r^{N-1}a_\t(|u'_\t(r)|^2)|u'_\t(r)|^2\, dr
%=\int_0^{+\infty}r^{N-1}g(u_\t(r))u_\t(r)\, dr,
%\]
%reaching a contradiction.
\\
Let us prove the following

\medspace \noindent
{\sc Claim:} there exists $C>0$ such that 
\begin{equation}\label{claim}
|a_\t(|u'_\t(r)|^2)u'_\t(r)|\le C, \qquad \hbox{for any $r\ge 0$ and $\t\in (0,\t_1]$}.
\end{equation}

\medspace

\noindent By the regularity of $u_\t$, we infer that $u_\t'(0)=0$ and so also 
\[
a_\t(|u'_\t(0)|^2)u'_\t(0)=0.
\]
We now consider the case $r>0$. 
Integrating the equation \eqref{eqradz}, for any $r>0$, we have
\begin{equation*}
-a_\t(|u'_\t(r)|^2)u'_\t(r)
=\frac1{r^{N-1}}\int_0^r s^{N-1}g(u_\t(s))\, ds.
\end{equation*}
By Lemma \ref{le:strauss} and by \eqref{unifbddz}, we deduce that there exists $R>1$, such that 
\begin{equation}\label{stimaunifz}
 |u_\t(r)|\le \bar c_2, \quad\hbox{for any  $\t\in (0,\t_1]$ and for any $r>R$,} 
\end{equation}
where $\bar c_2$ is  given in \eqref{gbarcz}. 
\\
By the continuous embedding of $\X_0$ in $L^p(\RN)$, for $p\in [2^*,q^*]$, and \eqref{unifbddz},  there exists $C>0$ such that $|u_\t|_p\le C\|u_\t\|_0\le C$, for $p\in [2^*,q^*]$ and any $\t\in (0,\t_1]$.  So, using \eqref{gbarczinf}, we have that, for any $0<r\le R$ and $\t\in (0,\t_1]$,  
\begin{equation*}
|a_\t(|u'_\t(r)|^2)u'_\t(r)|
\le \frac1{r^{N-1}}\int_0^r s^{N-1}|g(u_\t(s))|\, ds\le C.
\end{equation*}
While, for any $r> R$, 
\begin{align*}
|a_\t(|u'_\t(r)|^2)u'_\t(r)|
&\le \frac1{r^{N-1}}\int_0^r s^{N-1}|g(u_\t(s))|\, ds
\\
&\le \frac1{r^{N-1}}\left(\int_0^R s^{N-1}|g(u_\t(s))|\, ds
+\int_R^r s^{N-1}|g(u_\t(s))|\, ds\right)\\
&\le \frac C{r^{N-1}}
+\underbrace{\frac{c_1}{r^{N-1}}\int_1^r s^{N-1}|g(u_\t(s))|\, ds}_{(A)}.
\end{align*}
We have to estimate $(A)$. First of all, by Lemma \ref{le:strauss} and \eqref{unifbddz}, for $r>1$, we have that
\[
|u_\t(r)| \le C r^{-\frac{N-2}2} |\n u_\t|_2\le \bar C r^{-\frac{N-2}2}.
\]
Hence, by \eqref{stimaunifz} and \eqref{gbarczinf}, since $2^*<\g< q^*$,
\begin{align*}
(A)&
\le \frac{C}{r^{N-1}}\int_1^r s^{N-1}\big(|u_\t(s)|^{\g-1}+|u_\t(s)|^{q^*-1}\big)\, ds
\\
&\le \frac{ C}{r^{N-1}}\int_1^r s^{N-1-\frac{N-2}2(\g-1)}\, ds
\le C \left(r^{1-\frac{N-2}2(\g-1)}+1\right)\le C.
\end{align*}
Therefore the claim  is proved.
\\
Now we conclude as in the previous section.
\end{proof}

\end{document}